\documentclass[11pt]{article}
\usepackage{graphicx}
\usepackage{amsfonts,amsmath,amsthm,amstext,amssymb,url}
\usepackage{color}
\usepackage[dvipsnames]{xcolor}
\usepackage{authblk}
\usepackage{hyperref}

\def\I{\mathbb{I}}

\def\R{\mathbb{R}}

\def\var{\mathrm{Var}}
\def\eqdef{\buildrel{\rm def}\over =}

\newtheorem{theorem}{\sffamily Theorem}

\newtheorem{corollary}[theorem]{\sffamily Corollary}
\newtheorem{proposition}[theorem]{\sffamily Proposition}
\newtheorem{definition}[theorem]{\sffamily Definition}
\newtheorem{remark}[theorem]{\sffamily Remark}

\title{Eigenvalues of Real Matrices with Prescribed Principal Minors Sign and Descartes Law of Signs}
\begin{document}

\author[1]{Laureano Gonz\'alez--Vega}
\author[2]{J. Rafael Sendra}
\author[1]{Juana Sendra}
\affil[1]{Department of Quantitative Methods, CUNEF Universidad, Madrid, Spain}
\affil[2]{Universidad de Alcal\'a, Dpto. de F\'{\i}sica y Matem\'aticas, Alcal\'a de Henares, Madrid, Spain}
\date{}

\maketitle

\begin{abstract}
In this paper, using Descartes law of sign, we provide elementary proof of  results on the number of real eigenvalues of real matrices of which certain properties  on the signs of their principal minors are known.  More precisely, we analyze $P, N, Q, R, PN$ and $QR$ matrices as well as their variants \textit{almost}, \textit{weak} and \textit{sub-zero} matrices. 
\end{abstract}

%%%%%   INTRODUCTION %%%%%%%
\section*{Introduction}

In this paper we deal with matrices whose principal minors have signs that follow certain rule. Possibly the best known examples of this type of matrices are the $P$-matrices, where all principal minors are positive, or the $N$-matrices, where all principal minors are negative, see e.g. \cite{ChT}, \cite{HK}, \cite{Inada}, \cite{OPR}, \cite{PR}, \cite{T2002}, \cite{T2017}. 

As variants of   the $P$-matrices and  the $N$-matrices  appear, among others, the  so called almost $P$,  almost  $N$,   weak $P$,   weak $N$,    $P_0$--matrices, $N_0$--matrices, or the $PN$-matrices (see e.g. \cite{ChT}, \cite{CPS}, \cite{HK}, \cite{OPR}, \cite{PR}). Additionally, in the context of the linear complementarity problem, the $Q$ and $Q_0$--matrices also appear (see \cite{HK}). Complementary, in this paper we naturally extend the variants \textit{ almost, weak} and \textit{sub-zero} to the case of  $Q$ and $PN$--matrices. Likewise, as a generalization of the $Q$-matrices and the $PN$-matrices, we introduce the notions of $R$-matrices and $QR$-matrices as well as their variants.

These types of matrices appear in many applications as, for instance, the linear complementarity problem, global univalence of maps, dynamical system theory,  interval matrices and computational complexity, economics, graph theory, etc., (see e.g. \cite{ChT}, \cite{CPS}, \cite{T2002}, \cite{T2017}) One of the main keys to this applicability  is the behavior of their real eigenvalues. In general, the proofs that appear in the literature for the study of the number of eigenvalues (positive or negative) use more or less advanced results of linear algebra. We, in this work, taking into account the relationship that exists between the principal minors and the coefficients of the characteristic polynomial   and, consequently, the information that can be derived from it about their signs, we see how the direct application of Descartes  rule of signs allows to carry out, through elementary proofs, an analysis of the number of eigenvalues; this is the main contribution of this paper.

Throughout this paper, whenever we refer to the number of roots of a polynomial, or  to the number of eigenvalues of a matrix, we consider that the multiplicities have been taken into account. 

\section{Preliminaries}

In this section we recall Descartes law of signs and we introduce the type of matrices we analyze in the paper.
\subsection{Descartes Law of Signs}\label{subsec:descartes}

For a sequence of real numbers $b_0,b_1,\ldots,b_n$, $\var(b_0,b_1,\ldots,b_n)$ will denote the number of sign changes in $b_0,b_1,\ldots,b_n$ after dropping the zeros in the sequence. For a real number $a$, ${\rm sign}(a)$ will denote the sign of $a$ according to the following rule 
$${\rm sign}(a)=
\begin{cases}
+1& \hbox{if $a>0$},\\
\hfill0& \hbox{if $a=0$},\\
-1& \hbox{if $a<0$}.\\
\end{cases}
$$
Moreover, for a  non-constant polynomial $C(x)=a_0 x^n+a_1 x^{n-1}+\cdots+a_{n-1}x+a_n \in \R[x]$ we denote by  $$\var(C(x))\eqdef\var(a_0,a_{1},\ldots,a_n).$$
In this situation, the well-known Descartes law of signs is as follows;  see Section 2.2.1 in \cite{BPR} for a proof of the next two propositions.

\begin{proposition}\label{descartes}(Descartes Law of Signs)\hskip 0pt\\
Let  
$ C(x)=a_0 x^n+a_1 x^{n-1}+\cdots+a_{n-1}x+a_n\in \R[x], $
where  $a_0$ is nonzero. The number of positive real roots of $C(x)$ %, counted with multiplicities, 
is equal to $\var(C(x))-2k$, for some non--negative integer $k$.
\end{proposition}

\begin{proposition}\label{realr}\hskip 0pt\\
Let $C(x)=a_0 x^n+a_1 x^{n-1}+\cdots+a_{n-1}x+a_n\in \R[x],$
where   $a_0$ is nonzero. If all roots of $C$ are real then the number of positive real roots of $C(x)$ %, counted with multiplicities, 
is equal to $\var(C(x))).$
\end{proposition}

 Next corollary shows several consequences of these two propositions  that we will use very often in what follows.

\begin{corollary}\label{descartesconseq}\hskip 0pt\\
Let  \[C(x)=a_0 x^n+a_1 x^{n-1}+\cdots+a_{n-1}x+a_n\in \R[x], \]
where $a_0$ is nonzero. Then it holds that:
\begin{enumerate}
\item If $\var(C(x))=1$ then $C(x)$ has exactly one positive real root . %(counted with multiplicities).
\item If $\var(C(x))=0$ then $C(x)$ has no positive real roots.
%\textcolor{red}{\item If ${\bf Var}(C(-x))=2$ then $C(x)$ has either zero or  two positive real eigenvalues.}
\item If $\var(C(-x))=1$ then $C(x)$ has exactly one negative real root %(counted with multiplicities).
\item If $\var(C(-x))=0$ then $C(x)$ has no negative real roots.
%\textcolor{red}{\item If ${\bf Var}(C(-x))=2$ then $C(x)$ has either zero or  two  negative real eigenvalues.}
\end{enumerate}
\end{corollary}

\subsection{Type of Matrices}
Let us introduce the type of matrices we deal with in this paper. As we have mentioned in the introduction, most of the notions appear in
the literature. In addition, mimicking the requirements that are required, we introduce the notions of $R$ and $QR$--matrices with their almost, weak, and sub-zero alternatives.

For this purpose, we denote by $\R^{n\times n}$ the set of all $n\times n$ real matrices. Also, let $\mathbb{I}_n$  denote the  $n\times n$ identity. 
We recall that a principal minor is  the
determinant of a submatrix obtained from A when the same set of rows and columns
are stricken out.

\begin{definition}[Case $P$--matrix] Let $A\in \R^{n\times n}$. We say that: 
\begin{enumerate}
\item $A$ is a \textsf{$P$--matrix} if all its   principal minors are  positive.
\item $A$ is an \textsf{almost $P$--matrix} if its determinant is negative and all other principal minors are positive.
\item $A$ is a \textsf{weak $P$--matrix} if its determinant is positive and all other principal minors are non negative.
\item $A$ is a \textsf{$P_0$--matrix} if  all other principal minors are non negative.

\end{enumerate}
\end{definition}

\begin{definition}[Case $N$--matrix]   Let $A\in \R^{n\times n}$. We say that: 
\begin{enumerate}
\item $A$ is a \textsf{$N$--matrix} if all its   principal minors are  negative.
\item $A$ is an \textsf{almost $N$--matrix} if its determinant is positive and all other principal minors are negative.
\item $A$ is a \textsf{weak $N$--matrix} if its determinant is negative and all other principal minors are non positive.
\item $A$ is a \textsf{$N_0$--matrix} if all other principal minors are non positive.
\end{enumerate}
\end{definition}

\begin{definition}[Case $PN$--matrix]  Let $A\in \R^{n\times n}$. We say that: 
\begin{enumerate}
\item $A$ is a \textsf{$PN$--matrix} if the sign of each $k\times k$ principal minor, $k\in \{1,\ldots,n\}$, is  $(-1)^{k-1}$.
\item $A$ is an \textsf{almost $PN$--matrix} if the sign of its determinant  is $(-1)^{n}$, and  for $k\in \{1,\ldots,n-1\}$, the sign of each $k\times k$ principal minor is  $(-1)^{k-1}$.
\item $A$ is a \textsf{weak $PN$--matrix}  if the sign of its determinant  is $(-1)^{n-1}$, and 
 the sign of each $k\times k$ principal minor, $k\in \{1,\ldots,n-1\}$, is  either  $(-1)^{k-1}$ or $0$.
 \item $A$ is a \textsf{$(PN)_0$--matrix}  if 
 the sign of each $k\times k$ principal minor, $k\in \{1,\ldots,n\}$, is  either  $(-1)^{k-1}$ or $0$.
 \end{enumerate}
\end{definition}

\begin{definition} [Case $Q$--matrix] Let $A\in \R^{n\times n}$. We say that: 
\begin{enumerate}
\item $A$ is a \textsf{$Q$--matrix } if, for each $k\in \{1,\ldots,n\}$, the sum of all $k\times k$ principal minors is positive.
\item $A$ is an \textsf{almost $Q$--matrix} if  its determinant is negative and, for each $k\in \{1,\ldots, n-1\}$,   the sum of all $k\times k$ principal minors is positive.
\item $A$ is a \textsf{ weak $Q$-matrix } if  its determinant is positive and, for each $k\in \{1,\ldots, n-1\}$ the   sum of all $k\times k$ principal minors is non negative.
\item $A$ is a \textsf{$Q_0$-matrix } if  for each $k\in \{1,\ldots, n\}$ the   sum of all $k\times k$ principal minors is non negative.
\end{enumerate}
\end{definition}

\begin{definition} [Case $R$--matrix] Let $A\in \R^{n\times n}$. We say that: 
\begin{enumerate}
\item $A$ is an \textsf{$R$--matrix } if, for each $k\in \{1,\ldots,n\}$, the sum of all  $k\times k$ principal minors is negative.
\item $A$ is an \textsf{almost $R$--matrix} if  its determinant is positive and, for each $k\in \{1,\ldots, n-1\}$,   the sum of all $k\times k$ principal minors is negative.
\item $A$ is a \textsf{ weak $R$-matrix } if   its determinant is negative and, for each $k\in \{1,\ldots, n-1\}$ the   sum of all $k\times k$ principal minors is non positive.
\item $A$ is a \textsf{$R_0$-matrix } if   for each $k\in \{1,\ldots, n\}$ the   sum of all $k\times k$ principal minors is non positive.
\end{enumerate}
\end{definition}

\begin{definition} [Case $QR$--matrix] Let $A\in \R^{n\times n}$. We say that: 
\begin{enumerate}
\item $A$ is a \textsf{$QR$--matrix } if, for each $k\in \{1,\ldots, n\}$, the sum of all $k\times k$ principal minors is $(-1)^{k-1}$ .
\item $A$ is an \textsf{almost $QR$--matrix} if the sign of its determinant  is $(-1)^{n}$ and, for $k\in \{1,\ldots,n-1\}$, the sign of each $k\times k$ principal minor,  is  $(-1)^{k-1}$.
\item $A$ is a \textsf{ weak $QR$-matrix } if the sign of its determinant  is $(-1)^{n-1}$ and,  for $k\in \{1,\ldots,n-1\}$, the sign of the sum of all  $k\times k$ principal minor
is  either  $(-1)^{k-1}$ or $0$.
\item $A$ is a \textsf{$(QR)_0$-matrix } if   for $k\in \{1,\ldots,n\}$, the sign of the sum of all  $k\times k$ principal minor
is  either  $(-1)^{k-1}$ or $0$.
\end{enumerate}
\end{definition}

\begin{remark}\label{obs1} We observe that:
\begin{enumerate}
\item  $A$ is an almost $P$--matrix if and only if $A^{-1}$ is an $N$--matrix (see e.g \cite{OPR}).
\item Every $P$--matrix (resp.   almost $P$--matrix, weak $P$--matrix, $P_0$--matrix) is a  $Q$--matrix (resp.    almost $Q$--matrix, weak $Q$--matrix, $Q_0$--matrix).
\item  Every $N$--matrix (resp.   almost $N$--matrix, weak $N$--matrix, $N_0$--matrix) is a  $R$--matrix (resp.    almost $R$--matrix, weak $R$--matrix, $R_0$--matrix).
 \item Every $PN$--matrix (resp.   almost $PN$--matrix, weak $PN$--matrix) is a  $QR$--matrix (resp.    almost $QR$--matrix, weak $QR$--matrix, $(QR)_0$--matrix).
\end{enumerate}
\end{remark}

\section{Eigenvalues  Through Descartes Law of Signs}\label{sec:bohemian}

In this section, we show how Descartes law of signs provides automatically information about the real eigenvalues of $P$--matrices, $Q$-matrices, $N$--matrices,  $R$--matrices, $PN$--matrices and $QR$--matrices, and simplifies the usual proofs that are based on, for example for $N$--matrices, Perron-Frobenious Theorem and other results (see \cite{PR}). For the particular case of symmetric matrices we also show that the sign pattern of their eigenvalues is fixed and known in advance.
\vspace*{2mm}

Let $A\in \R^{n\times n}$, and let $C_A(\lambda)$ be its characteristic polynomial, namely,

\begin{equation}\label{eq:polcaracteristico}
C_A(\lambda)=\det(\lambda\I_n-A)=\lambda^n+a_1\lambda^{n-1}+\cdots+a_{n-1}\lambda+a_n.
\end{equation}

\noindent We recall that the coefficients of $C_A(\lambda)$ can be expressed by means of the principal minors. More precisely, it holds that 
\begin{equation}\label{eq:caracteristico}
a_k=(-1)^k    E_{k}(A)
\end{equation}
where   $E_{k}(A)$ denotes the sum of all $k\times k$ principal minors of $A$ (see e.g \cite{HJ}).

\begin{theorem}\label{DescartesQ}(Eigenvalues of $Q$--matrices)
\begin{enumerate}
\item If $A$ is either a $Q$--matrix or a  weak  $Q$--matrix or a   $Q_0$--matrix,  then   $A$ has no  negative real eigenvalues.
\item If $A$ is an almost $Q$--matrix then $A$ has exactly one  negative real eigenvalue. % (counted with multiplicities).
\item  If $A$ is an  $n\times n$ symmetric that is either a $Q$--matrix or  a weak $Q$--matrix then $A$ has $n$  positive real eigenvalues.  
%(counted with multiplicities).
\item  If $A$ is an  $n\times n$ symmetric  $Q_0$--matrix  then $A$ has $n$  non negative real eigenvalues. 
%  (counted with multiplicities).
\item If $A$ is an $n\times n$ symmetric almost $Q$--matrix, then $A$ has exactly one negative real eigenvalue and $n-1$  positive real eigenvalues.  
%(counted with multiplicities) 
\end{enumerate}
\end{theorem}

\begin{proof} Let   $C_A(\lambda)$ be as in \eqref{eq:polcaracteristico} and $E_{k}(A)$ be as in \eqref{eq:caracteristico}.
\begin{enumerate}
\item 
If  $A$ is a $Q$--matrix, for  all $k\in \{1,\ldots,n\}$  it holds that $E_{k}(A)>0$.  Furthermore, by \eqref{eq:caracteristico}, ${\rm sign}(a_k)=(-1)^{k}$
for $k\in\{1,\ldots,n\}$. Since
\begin{equation}\label{eq-CAmenos}
C_A(-\lambda)=(-1)^{n}\lambda^n+(-1)^{n-1}a_1\lambda^{n-1}+\cdots+(-1)^1a_{n-1}\lambda+a_n,
\end{equation}
the sequence of sign of the coefficients of $C_A(-\lambda)$ is
$$\begin{array}{l} \{(-1)^{n},(-1)^{n-1}(-1)^1,(-1)^{n-2}(-1)^2,\ldots,(-1)^1(-1)^{n-1},(-1)^{n}\}=\\
\{ (-1)^{n},(-1)^{n},\ldots,(-1)^{n},(-1)^{n}\}\end{array}$$
Thus $\var(C_A(-\lambda))=0$ and, according to Corollary \ref{descartesconseq} (4), we conclude that $A$ has  no   negative real eigenvalues.

If $A$ is either a weak $Q$--matrix or a $Q_0$ matrix, reasoning as above, the first   coefficient of $C_A(-\lambda)$ is $(-1)^{n}$, and all other coefficients have  sign either $(-1)^{n}$ or $0$. Then $\var(C_A(-\lambda))=0$, and
 the result follows from Corollary \ref{descartesconseq} (4).
\item If $A$ is an almost $Q$--matrix, reasoning as in the previous statement, the sequence of sign of the coefficients of $C_A(-\lambda)$ is
$$\{ (-1)^{n},(-1)^{n},\ldots,(-1)^{n},(-1)^{n+1}\}.$$
Thus $\var(C_A(-\lambda))=1$ and, according to Corollary \ref{descartesconseq} (3), we conclude that $A$ has exactly one negative real eigenvalue.%, counted with multiplicity.
\item Since $A$ is symmetric, by (1) we know that all eigenvalues are non negative. Moreover, since $A$ is either $R$ or weak $R$, the determinant of $A$ is non-zero and hence, by Statement (1), we conclude that $A$ has $n$  positive real eigenvalues. %  (counted with multiplicities).
\item It follows as in the previous case but taking into account that, since $A$ is $Q_0$, the determinant of $A$ can be zero.
\item It follows from the symmetry, from the fact that $\det(A)\neq 0$, and from statement (2).
\end{enumerate}
\end{proof}

The following corollary follows from Theorem \ref{DescartesQ} and Remark \ref{obs1}  (2).

\begin{corollary}\label{DescartesP}(Eigenvalues of $P$--matrices)
\begin{enumerate}
\item If $A$ is either a $P$--matrix or a  weak  $P$--matrix or a   $P_0$--matrix,  then   $A$ has no  negative real eigenvalues.
\item If $A$ is an almost $P$--matrix then $A$ has exactly one  negative real eigenvalue. % (counted with multiplicities).
\item  If $A$ is an  $n\times n$ symmetric that is either a $P$--matrix or  a weak $P$--matrix then $A$ has $n$  positive real eigenvalues.  
%(counted with multiplicities).
\item  If $A$ is an  $n\times n$ symmetric  $P_0$--matrix  then $A$ has $n$  non negative real eigenvalues. % (counted with multiplicities).
\item If $A$ is an $n\times n$ symmetric almost $P$--matrix, then $A$ has exactly one negative real eigenvalue and $n-1$  positive real eigenvalues.  %(counted with multiplicities).
\end{enumerate}

\end{corollary}

\begin{theorem}\label{DescartesR}(Eigenvalues of $R$--matrices)
\begin{enumerate}
\item If $A$ is either an $R$--matrix or a  weak  $R$--matrix, then $A$ has exactly one negative real eigenvalue. % (counted with multiplicities).
\item If $A$ is    $R_0$--matrix, then $A$ has at most one negative real eigenvalue. %(counted with multiplicities).
 \item If $A$ is an almost $R$--matrix then $A$ has either zero or  two  negative real  eigenvalues. %(counted with multiplicities).
 \item  If $A$ is an $n\times n$ symmetric that is either   $R$--matrix or    weak  $R$--matrix, then $A$ has $n-1$ positive real  eigenvalues and one negative eigenvalue. 
 \item  If $A$ is an $n\times n$ symmetric $R_0$--matrix, then $A$ has either
 $n$ non negative real  eigenvalues or   $n-1$  non negative real  eigenvalues and one negative eigenvalue.
\item If $A$ is an $n\times n$ symmetric almost $R$--matrix, then $A$ has either $n$  positive  real eigenvalues or $n-2$  positive real eigenvalues and two negative real eigenvalues.
\end{enumerate}
\end{theorem}

\begin{proof}  Let   $C_A(\lambda)$ be as in \eqref{eq:polcaracteristico} and $E_{k}(A)$ be as in \eqref{eq:caracteristico}.
\begin{enumerate}
\item 
If  $A$ is a $R$--matrix, for  all $k\in \{1,\ldots,n\}$  it holds that $E_{k}(A)<0$.  Furthermore,  by \eqref{eq:caracteristico}, ${\rm sign}(a_k)=(-1)^{k+1}$
for $k\in\{1,\ldots,n\}$. Taking into account \eqref{eq-CAmenos}, 
%$$C_A(-\lambda)=(-1)^{n}\lambda^n+(-1)^{n-1}a_1\lambda^{n-1}+\cdots+(-1)^1a_{n-1}\lambda+a_n,$$
the sequence of sign of the coefficients of $C_A(-\lambda)$ is
$$\begin{array}{l} \{(-1)^{n},(-1)^{n-1}(-1)^2,(-1)^{n-2}(-1)^3,\ldots,(-1)^1(-1)^{n},(-1)^{n+1}\}=\\
\{ (-1)^{n},(-1)^{n+1},\ldots,(-1)^{n+1},(-1)^{n+1}\}.\end{array}$$
Thus $\var(C_A(-\lambda))=1$ and, according to Corollary \ref{descartesconseq} (3), we conclude that $A$ has exactly one negative  real eigenvalue.

If $A$ is  a weak $R$--matrix, reasoning as above, the first   coefficient of $C_A(-\lambda)$ is $(-1)^{n}$, the last one is $(-1)^{n+1}$ because the determinant is non-zero, and all intermediate coefficients have  sign either $(-1)^{n+1}$ or $0$. So, $\var(C_A(-\lambda))=1$. Now the result follows from Corollary \ref{descartesconseq} (3).
\item If $A$ is an $R_0$--matrix, reasoning as above, the first   coefficient of $C_A(-\lambda)$ is $(-1)^{n}$, and all other coefficients have  sign either $(-1)^{n+1}$ or $0$. So, $\var(C_A(-\lambda))\in \{0,1\}$. Now the result follows from Corollary \ref{descartesconseq} (3), (4).
\item If $A$ is an almost $R$--matrix, reasoning as in the proof of   statement (1), the sequence of sign of the coefficients of $C_A(-\lambda)$ is
$$\{ (-1)^{n},(-1)^{n+1},\ldots,(-1)^{n+1},(-1)^{n}\}.$$
Thus, $\var(C_A(-\lambda))=2$. Now, according to Proposition \ref{descartes}, we conclude that $A$  has either zero or  two  negative real eigenvalues. % counted with multiplicities.
\item Since $A$ is symmetric,  all eigenvalues are real. Moreover, since $A$ is either $R$ or weak $R$, the determinant of $A$ is non-zero and hence, by Statement (1), we conclude that $A$ has $n-1$ positive real  eigenvalues and one negative eigenvalue. 
 %  (counted with multiplicities).
 \item It follows from Statement (2) taking into account that all eigenvalues are real and that the determinant can be zero.
  \item It follows from Statement (3) taking into account that all eigenvalues are real and that the determinant is not zero.
\end{enumerate}
\end{proof}

The following Corollary follows from Theorem \ref{DescartesR} and Remark \ref{obs1}  (3).

\begin{corollary}\label{DescartesN}(Eigenvalues of $N$--matrices)
\begin{enumerate}
\item If $A$ is either an $N$--matrix or a  weak  $N$--matrix, then $A$ has exactly one negative real eigenvalue. % (counted with multiplicities).
\item If $A$ is    $N_0$--matrix, then $A$ has at most one negative real eigenvalue. %(counted with multiplicities).
 \item If $A$ is an almost $N$--matrix then $A$ has either zero or  two  negative real  eigenvalues. %(counted with multiplicities).
 \item  If $A$ is an $n\times n$ symmetric that is either   $N$--matrix or    weak  $N$--matrix, then $A$ has $n-1$ positive real  eigenvalues and one negative eigenvalue. 
 \item  If $A$ is an $n\times n$ symmetric $N_0$--matrix, then $A$ has either
 $n$ non negative real  eigenvalues or   $n-1$  non negative real  eigenvalues and one negative eigenvalue.
\item If $A$ is an $n\times n$ symmetric almost $N$--matrix, then $A$ has either $n$  positive  real eigenvalues or $n-2$  positive real eigenvalues and two negative real eigenvalues.
\end{enumerate}
\end{corollary}
\begin{theorem}\label{DescartesQR}(Eigenvalues of $QR$--matrices)
\begin{enumerate}
\item If $A$ is  either an $QR$--matrix or a weak $QR$--matrix, then $A$ has exactly one positive real eigenvalue. % (counted with multiplicities).
\item If $A$ is a   $(QR)_0$--matrix then $A$ has  at most one positive real eigenvalue. %  (counted with multiplicities).
 \item If $A$ is an almost $QR$--matrix then $A$ has either zero or  two  positive real  eigenvalues. %  (counted with multiplicities).
\item  If $A$ is an $n\times n$ symmetric that is either   $QR$--matrix or    weak  $QR$--matrix, then $A$ has $n-1$ negative real  eigenvalues and one positive eigenvalue. 
 \item  If $A$ is an $n\times n$ symmetric $(QR)_0$--matrix, then $A$ has either
 $n$ non positive real  eigenvalues or   $n-1$  non positive real  eigenvalues and one positive eigenvalue.
\item If $A$ is an $n\times n$ symmetric almost $QR$--matrix, then $A$ has either $n$  negative  real eigenvalues or $n-2$  negative real eigenvalues and two positive real eigenvalues.
\end{enumerate}
\end{theorem}

\begin{proof}  Let   $C_A(\lambda)$ be as in \eqref{eq:polcaracteristico} and $E_{k}(A)$ be as in \eqref{eq:caracteristico}.
\begin{enumerate}
\item 
Since  $A$ is a $QR$--matrix, for  all $k\in \{1,\ldots,n\}$  it holds that  ${\rm sign} (E_{k}(A))= (-1)^{k-1}$.  So
$${\rm sign}(a_k)=(-1)^{k}(-1)^{k-1}=(-1)^{2k-1}=-1.$$
 Thus, the sequence of sign of the coefficients of $C_A(\lambda)$ is $$\{+1,-1,-1,\ldots,-1,-1\},$$
and $\var(C_A(\lambda))=1$. According to Corollary \ref{descartesconseq} (1), we conclude that $C(\lambda)$ has  exactly one real positive eigenvalue. %, counted with multiplicity.

If $A$ is a weak $QR$--matrix, then the first and the last elements in the sequence of sign of the coefficients of $C_A(\lambda)$ are $1$ and $-1$, respectively; while all the others are either $-1$ or $0$. Therefore, $\var(C_A(\lambda))=1$ and proof ends as above.
\item If $A$ is $(QR)_0$, reasoning as in the proof of the previous statement, the first  element in the sequence of sign of the coefficients of $C_A(\lambda)$ is $1$ and all the other are either $-1$ or $0$. Therefore $\var(C_A(\lambda))\in \{0,1\}$. Thus, applying Corollary \ref{descartesconseq} (1), one concludes the proof.
\item Let $A$ be an almost $QR$--matrix,  reasoning as in  the previous statements, the sequence of sign of the coefficients of $C_A(\lambda)$ is
$${\rm sign}(a_k)=-1\,\,\,\, \mbox{for}\,\, k\in\{1,\ldots,n-1\}$$
 and
$${\rm sign}(a_n)=(-1)^{2n}=+1.$$ Thus,
we have that the sign of the coefficients of $C(\lambda)$ is
$\{+1,-1,\ldots,-1,+1\}$
and therefore $\var(C_A(\lambda))=2$ and, according to Proposition \ref{descartes}, we conclude that $C(\lambda)$ has  no real positive roots or  exactly two real positive roots. %, counted with multiplicity.
\end{enumerate}
The proof of the Statements (4), (5), (6) are analogous to those in Theorem \ref{DescartesR}.
\end{proof}

The following Corollary follows from Theorem \ref{DescartesQR} and Remark \ref{obs1}  (4).

\begin{corollary}\label{DescartesPN}(Eigenvalues of  $PN$--matrices)
\begin{enumerate}
\item If $A$ is  either an $PN$--matrix or a weak $PN$--matrix, then $A$ has exactly one positive real eigenvalue. % (counted with multiplicities).
\item If $A$ is a   $(PN)_0$--matrix then $A$ has  at most one positive real eigenvalue. %  (counted with multiplicities).
 \item If $A$ is an almost $PN$--matrix then $A$ has either zero or  two  positive real  eigenvalues. %  (counted with multiplicities).
\item  If $A$ is an $n\times n$ symmetric that is either   $PN$--matrix or    weak  $PN$--matrix, then $A$ has $n-1$ negative real  eigenvalues and one positive eigenvalue. 
 \item  If $A$ is an $n\times n$ symmetric $(PN)_0$--matrix, then $A$ has either
 $n$ non positive real  eigenvalues or   $n-1$  non positive real  eigenvalues and one positive eigenvalue.
\item If $A$ is an $n\times n$ symmetric almost $PN$--matrix, then $A$ has either $n$  negative  real eigenvalues or $n-2$  negative real eigenvalues and two positive real eigenvalues.
\end{enumerate}\end{corollary}

In Table \ref{table1} we summarize all type of situations analyzed above.

\begin{table}[h]
\centering
\begin{tabular}{|c|c|c|}
  \hline
  % after \\: \hline or \cline{col1-col2} \cline{col3-col4} ...
  {\sf Type of matrix} & {\sf General case} & {\sf Symmetric case}   \\\hline 
   $\begin{array}{c} 
  \text{$Q$ or weak--$Q$ or }  \\
   \text{$P$ or weak--$P$}
   \end{array}$
    & $0$  negative & $n$   positive \\\hline
  $Q_0$ or $P_0$ & $0$    negative & $n$  non negative \\\hline
   $\begin{array}{c} 
   \text{almost $Q$ or} \\
   \text{almost $P$}
   \end{array}$ &  $1$ negative &  $n-1$ positive, and $1$ negative \\\hline
        \hline
    $\begin{array}{c} 
   \text{$R$ or weak--$R$ or } \\
   \text{$N$ or weak--$N$}
   \end{array}$
    & $1$  negative &  $n-1$ positive, and $1$ negative  \\\hline
  $R_0$ or $N_0$ &      $\leq 1$ negative &  $\begin{array}{c} \text{$n$ non negative, or} \\ \text{$n-1$ non negative and 1 negative}  \end{array}$   \\\hline
   $\begin{array}{c} 
   \text{almost $R$ or} \\
   \text{almost $N$}
   \end{array}$ &  $\begin{array}{c} \text{$0$ negative or} \\ \text{$2$ negative} \end{array}$ & $\begin{array}{c} \text{$n$ positive, or} \\ \text{$n-2$ positive and $2$ negative} \end{array}$ \\\hline
        \hline
    $\begin{array}{c} 
   \text{$QR$ or weak--$QR$ or } \\
   \text{$PN$ or weak--$PN$}
   \end{array}$
    & $1$  positive &  $n-1$ negative, and $1$ positive  \\\hline
  $(QR)_0$ or $(PN)_0$ &    $\leq 1$ positive &  $\begin{array}{c} \text{$n$ non positive, or} \\ \text{$n-1$ non positive and 1 positive}  \end{array}$   \\\hline
   $\begin{array}{c} 
   \text{almost $QR$ or} \\
   \text{almost $PN$}
   \end{array}$ &  $\begin{array}{c} \text{$0$ positive or} \\ \text{$2$ positive} \end{array}$ & $\begin{array}{c} \text{$n$ negative, or} \\ \text{$n-2$ negative and $2$ positive} \end{array}$ \\
      \hline
\end{tabular}\caption{Number of real eigenvalues of the different types of matrices}\label{table1}
\end{table}

%%%%%%%%%%%%%%%%%%%%%%%%%%%%%%%%%%%%%%%%%%%%%%%%%%%%%%
\section{\bf Acknowledgements. }
Authors are partially supported by the grant PID2020-113192GB-I00 (Mathematical Visualization: Foundations, Algorithms and Applications) from the Spanish MICINN.

%%%%%%%%%%%%%%%%%%%%%%%%%%%%%%%%%%%%%%%%%%%%%%%%%%%%%%

\end{document}